%% file: main.tex
\documentclass[12pt]{amsart}
\usepackage[margin=1.25in]{geometry}
\usepackage{latexsym}
\usepackage{amsmath}
\usepackage{amssymb}
\usepackage{amsthm}
\usepackage{stmaryrd}
\usepackage{amscd}
\usepackage{enumerate} 
\usepackage{amssymb} 
\usepackage{mathrsfs}
\usepackage[all]{xy}
\usepackage{color}
\usepackage{graphicx}
\usepackage{mathtools}
\usepackage{comment}
\usepackage{multirow}
\usepackage{hyperref}
\hypersetup{colorlinks=true}

\numberwithin{equation}{section}

\title{Extendability of differential forms via Cartier operators}
\author{Tatsuro Kawakami}
\email{tatsurokawakami0@gmail.com}
\address{Department of Mathematics, Graduate School of Science, Kyoto University, Kyoto 606-8502, Japan}

\def\phi{\varphi}
\def\epsilon{\varepsilon}

\def\log{\operatorname{log}}

\def\Spec{\operatorname{Spec}}

\def\Supp{\operatorname{Supp}}

\def\m{{\mathfrak m}}

\newcommand{\Q}{\mathbb{Q}}

\newcommand{\Z}{\mathbb{Z}}

\newcommand{\sO}{\mathcal{O}}

\newcommand{\mydot}{{{\,\begin{picture}(1,1)(-1,-2)\circle*{2}\end{picture}\ }}}

\theoremstyle{plain}
\newtheorem{thm}{Theorem}[section] 

\newtheorem{prop}[thm]{Proposition}

\newtheorem{lem}[thm]{Lemma}
\theoremstyle{definition} 
\newtheorem{defn}[thm]{Definition}
\newtheorem{conv}[thm]{Convention}
 
\newtheorem{eg}[thm]{Example} 

\theoremstyle{remark}
\newtheorem{rem}[thm]{Remark}

\newtheorem{defn and notation}[thm]{Definition and Notation}

\theoremstyle{plain}
\newtheorem{theo}{Theorem}


\begin{document}
\tolerance = 9999

\maketitle
\markboth{Tatsuro Kawakami}{Extension theorem via Cartier operators}

\begin{abstract}
Let $X$ be a normal variety over a perfect field of positive characteristic and $B$ a reduced divisor on $X$. 
We prove that if the Cartier isomorphism on the log smooth locus of $(X,B)$ extends to the entire $X$, then $(X,B)$ satisfies the logarithmic extension theorem for differential forms.
As an application, we show that the logarithmic extension theorem holds for good quotients of smooth varieties by actions of reduced linearly reductive group schemes. In addition, the logarithmic extension theorem for one-forms holds for singularities of higher codimension under assumptions about Serre's condition.
We also prove that tame quotients satisfy the regular extension theorem.
\end{abstract}

\section{Introduction}
Differential forms play an essential role in the study of algebraic varieties. 
Let $X$ be a normal variety over a perfect field, $f\colon Y\to X$ a resolution of singularities, and $j\colon U\hookrightarrow X$ the inclusion of the smooth locus.
It is important to ask if every $i$-form on $U$ extends to an $i$-form on $Y$.
This is equivalent to saying that the restriction map 
\[
f_{*}\Omega^{i}_Y\hookrightarrow \Omega^{[i]}_X\coloneqq j_{*}\Omega^{i}_{U}
\]
is an isomorphism, or that $f_{*}\Omega^{i}_Y$ is reflexive.
We often consider extendability admitting a logarithmic pole along the reduced $f$-exceptional divisor as well.
We define precise terminologies in a more general setting.

\begin{defn}[Regular and logarithmic extension theorem]\label{def:log ext thm}
Let $X$ be a normal variety over a perfect field, $B$ a reduced divisor on $X$, and $i\geq 0$ an integer. 
\begin{enumerate}
\item[\textup{(1)}] We say that $X$ satisfies the \textit{regular extension theorem for $i$-forms} if, for any proper birational morphism $f\colon Y\to X$ from a normal variety $Y$, the restriction map
\[
f_{*}\Omega^{[i]}_Y\hookrightarrow \Omega^{[i]}_X
\]
is an isomorphism.
\item[\textup{(2)}] We say that $(X,B)$ satisfies the \textit{logarithmic extension theorem for $i$-forms} if, for any proper birational morphism $f\colon Y\to X$ from a normal variety $Y$, the restriction map
\[
f_{*}\Omega^{[i]}_Y(\log f^{-1}_{*}B+E)\hookrightarrow \Omega^{[i]}_X(\log B)
\]
is an isomorphism, where $E$ is the reduced $f$-exceptional divisor. We refer to Subsection \ref{subsection:Notation and terminologies} for detailed notation.
\end{enumerate}
If the regular (resp.~logarithmic) extension theorem for $i$-forms holds for all $i\geq 0$, then we simply say the \textit{regular (resp.~logarithmic) extension theorem holds}. 
When we do not distinguish the regular and logarithmic extension theorems, we simply call them the \textit{extension theorem}.
The extension theorem is a local property on $X$.
\end{defn}

The study of the extension theorem has a long history for $X$ defined over the field of complex numbers.
We refer to \cite{Ohsawa,Flenner88,Graf21,SvS85} for singularities of higher codimension and \cite{GKK,GKKP,KS21,Namikawa} for singularities appearing in the minimal model program for example.

On the other hand, when $X$ is defined over a perfect field of characteristic $p>0$, very little are known.
Graf \cite{Gra} recently proved that the logarithmic extension theorem holds for two-dimensional log canonical singularities in $p>5$ through insightful analysis of the dual graphs of their minimal resolutions. However, due to the dependence on the classification of such singularities, a direct generalization to higher-dimensional cases appears challenging.

In this paper, our focus is on the extension theorem for higher-dimensional varieties over a perfect field of positive characteristic. Previous studies on the extension theorem relied on (mixed) Hodge theory, the minimal model program, the existence of log resolutions, and vanishing theorems. However, these geometric tools either cannot be assumed or are currently unknown for higher-dimensional varieties in positive characteristic.

To address this challenge, we propose a purely algebraic approach to the extension theorem in positive characteristic.

\subsection{Extendability of the Cartier isomorphism and differential forms}
Firstly, we provide a sufficient condition for the logarithmic extension theorem using the Cartier operator.
Let $X$ be a normal variety over a perfect field of positive characteristic, $B$ a reduced divisor on $X$, and $j\colon U\hookrightarrow X$ the inclusion of the log smooth locus $U$ of $(X,B)$.
We define locally free $\sO_U$-modules by
 \begin{align*}
    &B^{i}_U(\log B)\coloneqq \mathrm{Im}(F_{*}d \colon F_{*}\Omega^{i-1}_U(\log B)\to F_{*}\Omega^{i}_U(\log B))\,\,\text{and}\\
    &Z^{i}_U(\log B)\coloneqq \mathrm{Ker}(F_{*}d \colon F_{*}\Omega^{i}_U(\log B) \to F_{*}\Omega^{i+1}_U(\log B))
\end{align*}
for all $i\geq 0$, where $F_{*}d$ denotes the pushforward of the differential map by the Frobenius morphism.
There exists the exact sequence
\[
0 \to B^{i}_U(\log B) \to Z^{i}_U(\log B) \xrightarrow{C^{i}_{U, B}} \Omega^{i}_U(\log B)\to 0
\]
resulting from the $i$-th (logarithmic) Cartier isomorphism for all $i\geq 0$ (\cite[Theorem 7.2]{Kat70}).
The map $C^{i}_{U, B}$ is called the \textit{$i$-th Cartier operator}.
We define reflexive $\sO_X$-modules by $B_X^{[i]}(\log B)\coloneqq j_{*}B^{i}_U(\log B)$ and $Z_X^{[i]}(\log B)\coloneqq j_{*}Z^{i}_U(\log B)$.
The \textit{$i$-th reflexive Cartier operator}
\[
C^{[i]}_{X, B}\colon Z_X^{[i]}(\log B)\to \Omega_X^{[i]}(\log B)
\]
is defined as $j_{*}C^{i}_{U,B}$.
We prove that if the $i$-th Cartier isomorphism 
\[
C^{i}_{U, B}\colon Z^{i}_U(\log B)/B^{i}_U(\log B)\xrightarrow{\cong}\Omega^{i}_U(\log B)
\] on $U$ extends to $X$, then every logarithmic $i$-form on $U$
extends to a reflexive logarithmic $i$-form on any proper birational model $Y$ above $X$.

\begin{theo}[\textup{cf.~Theorem \ref{thm:generalized ext thm}}]\label{Introthm:ext thm}
Let $X$ be a normal variety over a perfect field of positive characteristic and $B$ a reduced divisor on $X$. 
We fix an integer $i\geq 0$.
If the $i$-th reflexive Cartier operator 
\[
C^{[i]}_{X, B}\colon Z_X^{[i]}(\log B)\to \Omega_X^{[i]}(\log B)
\]
is surjective, then $(X,B)$ satisfies the logarithmic extension theorem for $i$-forms.
\end{theo}
It is important to note that the proof of Theorem \ref{Introthm:ext thm} does not require the existence of resolutions of singularities. Additionally, the surjectivity of the first reflexive Cartier operator is also a necessary condition for the logarithmic extension theorem for one-forms when $X$ has a log resolution and only rational singularities (see Proposition \ref{prop:Converse direction}).

\subsection{\texorpdfstring{$F$}--liftability and quotients by linearly reductive group schemes}\label{subsection:finite corver}
A normal variety $X$ over a perfect field $k$ of positive characteristic is said to be \textit{locally $F$-liftable} if there exists an open covering $\{U_i\}_{i\in I}$ such that every $U_i$ admits a flat lift to the ring $W_2(k)$ of Witt vectors of length two, along with its Frobenius morphism (see Definition \ref{def:F-lift}).
For normal locally $F$-liftable varieties, the surjectivity of the reflexive Cartier operators can be verified, and consequently, they satisfy the logarithmic extension theorem by Theorem \ref{Introthm:ext thm} (see Theorem \ref{thm:ext thm for F-lift}).
Achinger--Witaszek--Zdanowicz \cite{AWZ2} proved that good quotients of smooth varieties by actions of reduced linearly reductive group schemes are $F$-liftable.
Furthermore, we will show in Theorem \ref{thm:F-liftable singularities} that quotient singularities by (possibly non-reduced) finite linearly reductive group schemes, which have been studied in detail by Liedtke--Martin--Matsumoto \cite{LMM}, are also $F$-liftable.
Summarizing, the following theorem holds.

\begin{theo}\label{Introthm:linearly reductive quotient}
Let $X$ be a normal variety over a perfect field $k$ of positive characteristic and $B$ a reduced divisor on $X$. 
If $(X,B)$ is locally $F$-liftable, then $(X,B)$ satisfies the logarithmic extension theorem. 
In particular, 
\begin{enumerate}
    \item[\textup{(1)}] a good quotient of a smooth variety over $k$ by an action of a reduced linearly reductive group scheme and
    \item[\textup{(2)}] a variety over $k$ with quotient singularities by (possibly non-reduced) finite linearly reductive group schemes
\end{enumerate}
satisfy the logarithmic extension theorem. 
\end{theo}

\begin{rem}\label{rem:linearly reductive quotient}\,
\begin{enumerate}
\item[\textup{(1)}] Two-dimensional $F$-regular singularities are quotient singularities by finite linearly reductive group schemes by \cite[Theorem 1.6]{LMM}.
 Therefore, they satisfy the logarithmic extension theorem by Theorem \ref{Introthm:linearly reductive quotient}.
\item[\textup{(2)}] In Theorem \ref{Introthm:linearly reductive quotient} (1) and (2), the assumption of the linearly reductive property is necessary. For instance, consider a rational double point of type $E_8^2$ in characteristic two, which is a quotient singularity by the cyclic group scheme $\mathbf{C}_2$. This singularity violates the logarithmic extension theorem, as proved in Example \ref{example:C_2-quotient}.
\item[\textup{(3)}] In Theorem \ref{Introthm:linearly reductive quotient} (1) and (2), the regular extension theorem cannot be expected. For example, consider a $\frac{1}{p}(1,1)$-singularity for every prime number $p$. This singularity is a quotient singularity by the finite linearly reductive group scheme $\boldsymbol{\mu}_p$. However, it violates the regular extension theorem, as shown in \cite[Proposition 1.5]{Langer19} or \cite[Example 10.2]{Gra}. Additionally, a weighted projective space $\mathbb{P}(1:1:p)$, which is a good quotient of $\mathbb{A}^{3}\setminus\{0\}$ by an action of the reduced linearly reductive group scheme $\mathbb{G}_m$, has a $\frac{1}{p}(1,1)$-singularity.
\item[\textup{(4)}] When the base field is the field of complex numbers, Heuver \cite[Theorem 1]{Heuver} proved that the regular extension theorem holds for a good quotient $X$ of a smooth variety by an action of a reductive group scheme when $\dim X\leq 4$. Subsequently, Kebekus--Schnell \cite[Corollary 1.7]{KS21} resolved it in all dimensions by proving a more general statement.
\end{enumerate}
\end{rem}

Graf \cite{Gra} has already proven that two-dimensional $F$-regular singularities satisfy the logarithmic extension theorem in a different way (see also \cite[Theorem 4.7]{Kaw1}).
In his paper, he explores the reasons behind the good behavior of $F$-regularity when discussing the logarithmic extension theorem. Theorems \ref{Introthm:ext thm} and \ref{Introthm:linearly reductive quotient} provide an answer to his question.

As mentioned in Remark \ref{rem:linearly reductive quotient} (3), the regular extension theorem cannot be expected in Theorem \ref{Introthm:linearly reductive quotient} (1) and (2).
However, it holds for quotients by finite group schemes order prime to $p$.
In fact, the following theorem holds.

\begin{theo}\label{Introthm:tame quotient}
Let $X$ be a normal variety over a perfect field of characteristic $p>0$.
If there exists a finite surjective morphism $g\colon X' \to X$ of degree prime to $p$ from a smooth variety $X'$, then $X$ satisfies the regular extension theorem. 
\end{theo}

\subsection{Singularities of higher codimension}
Next, we consider the case where the singular locus has higher codimension.
Let $X$ be a normal variety over a perfect field of positive characteristic and $B$ a reduced divisor on $X$
such that the non-log smooth locus of $(X, B)$ has codimension at least three.
In this case, we can show that the reflexive Cartier operator $C^{[i]}_{X,B}\colon Z_X^{[i]}(\log B)\to \Omega_X^{[i]}(\log B)$ is surjective if the reflexive sheaf $B^{[i]}_X(\log B)$ satisfies Serre's condition $(S_3)$.
Determining whether $B^{[i]}_X(\log B)$ satisfies $(S_3)$ is generally not easy.
However, when $i=1$, the reflexive sheaf $B^{[1]}_X(\log B)$ coincides with the cokernel of the Frobenius map on the smooth locus of $X$, and we can find certain conditions under which $B^{[1]}_X(\log B)$ satisfies $(S_3)$.

\begin{theo}\label{Introthm:Flenner type ext thm}
Let $X$ be a normal variety over a perfect field of positive characteristic and $B$ a reduced divisor on $X$. 
Suppose that the non-log smooth locus of $(X, B)$ has codimension at least three.
In addition, we assume that one of the following conditions holds:
\begin{enumerate}
    \item[\textup{(1)}] The singular locus of $X$ has codimension at least four and $X$ satisfies $(S_4)$.  
    \item[\textup{(2)}] $X$ is $F$-injective and satisfies $(S_3)$.
    \end{enumerate}
Then $(X,B)$ satisfies the logarithmic extension theorem for one-forms.
\end{theo}

For example, $F$-rational singularities satisfy the condition (2).
It is interesting to compare Theorem \ref{Introthm:Flenner type ext thm} with Graf's counterexamples to Flenner's extension theorem in positive characteristic, which show that Serre's condition is essential.

\begin{thm}[\textup{cf.~\cite[Theorem 3]{Graf21}, see also Example \ref{eg:counterex to Flenner:non-Cohen-Macaulay case}}]\label{thm:Graf's example}
Let $p$ be a prime number and $d\geq 3$ an integer.
Then there exists a normal variety $X$ over a perfect field of characteristic $p$ of $\dim X=d$ that has only an isolated singularity and violates the logarithmic extension theorem for one-forms.
\end{thm}

\subsection{Further discussion}
Finally, we explore an application of Theorem \ref{Introthm:ext thm} to singularities appearing in the minimal model program, where a connection to $F$-singularities exists, suggesting potential affirmative outcomes. Specifically, for two-dimensional klt singularities over a perfect field of characteristic $p > 5$, they are known to be $F$-regular (\cite[Theorem 1.1]{Hara(two-dim)}), and thus, they satisfy the logarithmic extension theorem by Theorem \ref{Introthm:linearly reductive quotient}.

Moreover, in dimension three, we obtain a partial affirmative result as follows: Instead of the $F$-injectivity condition in Theorem \ref{Introthm:Flenner type ext thm} (2), we can replace it with the condition that $X$ is quasi-$F$-split, a notion developed in \cite{KTTWYY1}.
In \cite{KTTWYY2}, it is proved that every three-dimensional $\Q$-factorial affine klt variety over a perfect field of characteristic $p > 41$ is quasi-$F$-split. 
From this, we can deduce that the logarithmic extension theorem for one-forms holds for three-dimensional terminal singularities over a perfect field of characteristic $p > 41$. Further details can be found in \cite{KTTWYY2}.

\section{Preliminaries}

\subsection{Notation and terminology}\label{subsection:Notation and terminologies}
Throughout the paper, we work over a fixed perfect field $k$ of characteristic $p>0$ unless stated otherwise.
A \textit{variety} means an integral separated scheme of finite type.
We say that a pair $(X,B)$ of a normal variety $X$ and a $\Q$-divisor $B$ on $X$ is \textit{log smooth} if $X$ is smooth and $B$ has simple normal crossing support. 
Given a normal variety $X$, a reduced divisor $B$ on $X$, a $\Q$-divisor $D$ on $X$, and an integer $i\geq 0$, we denote $j_{*}(\Omega_U^{i}(\log B)\otimes \sO_U(\lfloor D\rfloor))$ by $\Omega_X^{[i]}(\log B)(D)$, where $j\colon U\hookrightarrow X$ is the inclusion of the log smooth locus $U$ of $(X, B)$.

\subsection{Serre's condition and reflexive sheaves}
In this subsection, we recall basic facts about Serre's condition and reflexive sheaves. 

\begin{defn}
Let $X$ be a Noetherian scheme, $\mathcal{F}$ a coherent $\sO_X$-module, and $n\geq 0$ an integer.
We say that $\mathcal{F}$ satisfies \textit{Serre's condition} $(S_n)$ if \[\mathrm{depth}_{\m_x}(\mathcal{F}_{x})\geq \min\{n, \dim(\Supp(\mathcal{F}_x))\}\] holds for every point $x\in X$.
\end{defn}

\begin{defn}
Let $X$ be a Noetherian integral scheme and $\mathcal{F}$ a coherent $\sO_X$-module.
We say that $\mathcal{F}$ is \textit{reflexive} if the natural map 
\[
\mathcal{F}\to \mathcal{F}^{**}\coloneqq \mathcal{H}\! \mathit{om}_{\sO_X}(\mathcal{H}\! \mathit{om}_{\sO_X}(\mathcal{F}, \sO_X), \sO_X)
\] is an isomorphism.
\end{defn}

When $X$ is a Noetherian normal irreducible scheme, a coherent $\sO_X$-module $\mathcal{F}$ with $\Supp(\mathcal{F})=X$ satisfies $(S_2)$ if and only if it is reflexive. 

\subsection{\texorpdfstring{$F$}--singularities and quotient singularities}
In this subsection, we gather facts about $F$-injectivity, $F$-regularity, $F$-liftability, and quotient singularities.

\begin{defn}\label{def:F-pure}
Let $X$ be a variety and $x\in X$ a point. We say that $X$ is \textit{$F$-injective at $x$} if the map
\[
H^i_{\m_x}(\sO_{X, x})\xrightarrow{F} H^i_{\m_x}(\sO_{X, x})
\] 
induced by the Frobenius morphism is injective for all $i\geq 0$.
We say that $X$ is \textit{$F$-injective} if $X$ is $F$-injective at every point of $X$.
\end{defn}

\begin{rem}
An $F$-injective singularity is closely related to a Du Bois singularity. In particular, Schwede \cite{Schwede} proved that if the modulo $p$-reduction of a reduced scheme of finite type over a field of characteristic zero is $F$-injective for dense $p$, then it is Du Bois.
\end{rem}

\begin{defn}\label{def:str F-reg}
Let $X$ be a variety and $x\in X$ a point.
We say that $X$ is \textit{strongly $F$-regular at $x$} if, for every non-zero element $c\in\sO_{X,x}^{\circ}$, there exists an integer $e>0$ such that the map
\[
\sO_{X,x}\xrightarrow{F^e} F^{e}_{*}\sO_{X,x}\xrightarrow{\times F_{*}^{e}c}F^{e}_{*}\sO_{X,x}
\] splits as an $\sO_{X, x}$-module homomorphism.
We say that $X$ is \textit{strongly $F$-regular} if $X$ is strongly $F$-regular at every point of $X$.
\end{defn}

\begin{defn}\label{def:F-lift}
Let $X$ be a normal variety and $B=\sum_{r=1}^n B_r$ a reduced divisor on $X$, where every $B_r$ is an irreducible component.  
We denote the ring of Witt vectors of length two by $W_2(k)$.
We say that $(X,B)$ is \textit{$F$-liftable} 
if there exist 
\begin{itemize}
	\item a flat morphism $\widetilde{X} \to \Spec W_2(k)$ together with a closed immersion $i\colon X\hookrightarrow \widetilde{X}$, 
	 \item a closed subscheme $\widetilde{B}_r$ of $\widetilde{X}$ flat over $W_2(k)$ for all $r\in\{1,\ldots,n\}$, and 
	 \item a morphism $\widetilde{F}\colon \widetilde{X}\to \widetilde{X}$ over $W_2(k)$
\end{itemize}
	such that 
\begin{itemize}
	\item the induced morphism $i\times_{W_2(k)}k \colon X\to \widetilde{X}\times_{W_2(k)} k$ is an isomorphism, 
	\item $(i\times_{W_2(k)}k)(B_r)= \widetilde{B}_r\times_{W_2(k)} k$ for all $r\in\{1,\ldots,n\}$, and 
    \item $\widetilde{F}\circ i=i \circ F$ and $\widetilde{F}^{*}\widetilde{B}_r=p\widetilde{B}_r$ for all $r\in\{1,\ldots,n\}$.
\end{itemize}

We say that $(X,B)$ is \textit{locally $F$-liftable} if there exists an open covering $\{U_i\}_{i\in I}$ of $X$ such that $(U_i,B|_{U_i})$ is $F$-liftable for every $i\in I$.
\end{defn}

\begin{rem}\label{rem:F-liftable}\,
\begin{enumerate}
    \item[\textup{(1)}] By \cite[Lemma 8.14]{EV92}, if $(X,B)$ is log smooth, then $(\widetilde{X}, \sum_{r=1}^{n}\widetilde{B}_r)$ is log smooth over $W_2(k)$.
    \item[\textup{(2)}] The $F$-liftability can be defined for schemes that are not necessarily of finite type in a similar way (see \cite[Section 3]{AWZ}). 
    \item[\textup{(3)}] Let $X$ be a normal toric variety and $B$ a reduced torus invariant divisor on $X$. Then $(X,B)$ is $F$-liftable.
When $B=0$, this is \cite[Example 3.1.2]{AWZ}. 
Even when $B\neq 0$, the $F$-liftability of $(X,B)$ can be verified by the fact that every torus invariant prime divisor corresponds to a one-dimensional cone.
\end{enumerate}
\end{rem}

\begin{defn}\label{def;good quotient}
Let $X'$ be a variety and $G$ a group scheme acting on $X'$.
We say that a morphism $g\colon X'\to X$ to a variety $X$ is a \textit{good quotient of $X'$ by the action of $G$} if 
\begin{enumerate}
\item[\textup{(1)}] $g$ is an affine morphism,
\item[\textup{(2)}] $g$ is $G$-invariant, and
\item[\textup{(3)}] the natural map $\sO_X\to (g_{*}\sO_{X'})^{G}$ is an isomorphism.
\end{enumerate}
\end{defn}

\begin{defn}
Let $G$ be an affine group scheme of finite type.
We say that $G$ is a \textit{linearly reductive group scheme} if every linear representation of $G$ is semi-simple.
\end{defn}

\begin{thm}\label{thm:AWZ}
Let $g\colon X'\to X$ be a good quotient of a smooth variety by an action of a reduced linearly reductive group scheme.
Then $X$ is normal and locally $F$-liftable.
\end{thm}
\begin{proof}
We may assume that $X$ is affine. Then $X'$ is also affine since $g$ is an affine morphism.
By the linearly reductive property, the induced map $\sO_X\to g_{*}\sO_{X'}$ splits, and we can deduce the normality of $X$ from that of $X'$.
Since every smooth affine variety is $F$-liftable by \cite[Example 3.1.1]{AWZ}, the $F$-liftability of $X$ follows from \cite[Theorem 2.10 (c)]{AWZ2}.   
\end{proof}

\begin{defn}\label{def:quotient singularity}
Let $k$ be an algebraically closed field of positive characteristic.
Let $X$ be a variety over $k$ and $G$ a finite group scheme over $k$. 
We say that a closed point $x\in X$ is a \textit{quotient singularity by $G$} 
if there exists a faithful action of $G$ on $\Spec k\llbracket x_1,\ldots,x_d\rrbracket$ fixing the closed point such that $\sO_{X,x}^{\wedge}\cong k\llbracket x_1,\ldots,x_d\rrbracket^{G}$.
If the $G$-action is free outside the closed point in addition, then we say that $x\in X$ is an \textit{isolated quotient singularity by $G$}.

We say that a variety over a perfect field $k$ has \textit{only (isolated) quotient singularities} if the base change $X_{\overline{k}}$ has only (isolated) quotient singularity, where $\overline{k}$ denotes the algebraic closure of $k$.
\end{defn}

\begin{thm}\label{thm:F-liftable singularities}
Let $X$ be a variety.
Suppose that one of the following conditions holds:
\begin{enumerate}
    \item[\textup{(1)}] $X$ has only quotient singularities by finite linearly reductive group schemes. 
    \item[\textup{(2)}] $X$ is $F$-regular and $\dim X=2$.
    \item[\textup{(3)}] $X$ is $F$-regular and has only isolated quotient singularities.
\end{enumerate}
Then $X$ is normal and locally $F$-liftable.
\end{thm}
\begin{proof}
We may assume that $X$ is affine.
Suppose that (1) holds. 
By \cite[Lemma 3.3.5]{AWZ} and \cite[Lemma A.11]{Zda18}, if $g\colon X'\to X$ is an \'etale surjective morphism from a normal variety $X'$, then $X$ is $F$-liftable if and only if so is $X'$.
Thus, by replacing $k$ with its algebraic closure, we may assume that $k$ is an algebraically closed field. 

Let $x\in X$ be a closed point.
Then there exists a faithful action of a finite linearly reductive group scheme $G$ on $\Spec k\llbracket x_1,\ldots,x_d\rrbracket$ fixing the closed point such that $\sO_{X,x}^{\wedge}\cong k\llbracket x_1,\ldots,x_d\rrbracket^{G}$.
Since $G$ is linearly reductive, the $G$-action is linearizable, and we have a linear $G$-action on $\mathbb{A}_{k}^{d}=\Spec k[x_1,\ldots, x_d]$ such that $\sO_{X,x}^{\wedge}$ is isomorphic to the completion of $k[x_1,\ldots, x_d]^G$ at the origin $x'$ (see \cite[Proof of Corollary 1.8]{Satriano} or \cite[Subsection 6.2]{LMM} for example).
By the Artin approximation (\cite[Corollary 2.6]{Artin(approximation)}), there exist \'etale morphisms $\pi\colon Z\to X$ and $\pi'\colon Z\to \mathbb{A}^d_k/G$ and a closed point $z\in Z$ such that $\pi(z)=x$ and $\pi'(z)=x'$.
Therefore, it suffices to show that $\mathbb{A}^d_k/G$ is $F$-liftable. 

Let $G^{\circ}$ be the connected component containing the identity.
Then $G^{\circ}$ is a normal subgroup scheme, and thus linearly reductive by \cite[Lemma 2.2]{Hashimoto}.
By \cite[Theorem 2.8 and Lemma 2.5 (1)$\Rightarrow$(3)]{Hashimoto}, the action of $G^{\circ}$ is diagonalizable, and thus the natural torus action on $\mathbb{A}_{k}^{d}$ descends to $\mathbb{A}_{k}^{d}/G^{\circ}$. Thus, $\mathbb{A}_{k}^{d}/G^{\circ}$ is normal and toric, and in particular, $F$-liftable by Remark \ref{rem:F-liftable} (3).
Since $G/G^{\circ}$ is \'etale, 
we conclude from Theorem \ref{thm:AWZ} that $\mathbb{A}_{k}^{d}/G=(\mathbb{A}_{k}^{d}/G^{\circ})/(G/G^{\circ})$ is $F$-liftable.

Finally, if (2) or (3) holds, then the assertion follows from (1) and \cite[Theorem 1.6 and Proposition 7.2 (3)$\Rightarrow$(1)]{LMM}.
\end{proof}

\begin{rem}
There is a notion of \textit{weak $F$-regularity} that is conjectured to coincide with strong $F$-regularity. In fact, for normal varieties that are $\Q$-Gorenstein or have dimension at most three, it is known that these notions coincide (see \cite[Chapter 3]{Maccrimmon} and \cite{Williams}).

By \cite[Proposition 7.1 (2)]{LMM}, isolated quotient singularities by finite linearly reductive group schemes are $\Q$-factorial. Therefore, we simply refer to $X$ as \textit{$F$-regular} in Theorem \ref{thm:F-liftable singularities} (2) and (3).
\end{rem}

\section{Reflexive Cartier operator}\label{sec:reflexive Cartier operators}
In Subsection \ref{Sec:Def of reflexive Cartier operators}, we define the reflexive Cartier operator.
The reflexive Cartier operator is not always surjective. In Subsection \ref{Sec:Suj of reflexive Cartier operators}, we investigate when it is surjective.

\subsection{Definition of the reflexive Cartier operator}\label{Sec:Def of reflexive Cartier operators}
Throughout this subsection, we use the following convention.

\begin{conv}
Let $X$ be a normal variety, $B$ a reduced divisor on $X$, and $D$ a $\Q$-divisor on $X$ such that the support $\Supp(\{D\})$ of the factional part is contained in $B$.
Set $D'\coloneqq \lfloor pD\rfloor-p\lfloor D\rfloor$.
Then $D'$ is a $\Z$-divisor such that $\Supp(D')\subset B$ and every coefficient of $D'$ is less than $p$.
Let $U$ be the log smooth locus of $(X,B)$ and $j\colon U\hookrightarrow X$ the inclusion.
\end{conv}

We recall the Cartier operator generalized by Hara \cite[Section 3]{Hara98}.
Let $j'\colon U\setminus B\hookrightarrow U$ be the inclusion.
Viewing $\Omega^{i}_U(\log B)(D')$ as a subsheaf of $(j')_{*}\Omega^{i}_{U\setminus B}$, we obtain the complex
\[
\Omega_U^{\mydot}(\log B)(D')\colon \sO_U(D') \xrightarrow{d} \Omega^1_U(\log B)(D') \xrightarrow{d} \cdots.
\]
Taking the Frobenius pushforward and tensoring with $\sO_U(D)=\sO_U(\lfloor D \rfloor)$, we obtain a complex
\[
F_{*}\Omega_U^{\mydot}(\log B)(pD) \colon F_{*}\sO_U(pD) \xrightarrow{F_{*}d\otimes\sO_U(D)}
F_{*}\Omega^1_U(\log B)(pD) \xrightarrow{F_{*}d\otimes\sO_U(D)} \cdots
\]
of $\sO_U$-modules, where 
\[F_{*}\Omega^{i}_U(\log B)(pD)\coloneqq F_{*}(\Omega^{i}_U(\log B)\otimes \sO_U(\lfloor pD\rfloor ))
\]
for all $i\geq 0$. We define coherent $\sO_U$-modules by 
\begin{align*}
    &B^{i}_U((\log B)(pD))\coloneqq \mathrm{Im}(F_{*}d\otimes\sO_U(D) \colon F_{*}\Omega^{i-1}_U(\log B)(pD) \to F_{*}\Omega^{i}_U(\log B)(pD)),\\
    &Z^{i}_U((\log B)(pD))\coloneqq \mathrm{Ker}(F_{*}d\otimes\sO_U(D) \colon F_{*}\Omega^{i}_U(\log B)(pD) \to F_{*}\Omega^{i+1}_U(\log B)(pD)),
\end{align*}
for all $i\geq 0$.

\begin{lem}\label{lem:Cartier isomorphism (log smooth case)}
There exist the following exact sequences:
\begin{equation}
    0 \to Z^{i}_U((\log B)(pD)) \to F_{*}\Omega^{i}_U(\log B)(pD) \to B^{i+1}_U((\log B)(pD))\to 0,\label{log smooth 1}
\end{equation}
\begin{equation}
    0 \to B^{i}_U((\log B)(pD))\to Z^{i}_U((\log B)(pD))\xrightarrow{C^{i}_{U,B}(D)} \Omega^{i}_U(\log B)(D)\to 0,\label{log smooth 2}
\end{equation}
for all $i\geq0$.
The map $C^{i}_{U,B}(D)$ coincides with $C^{i}_{U,B}\otimes \sO_U(D)$ on $U\setminus \Supp(D')$, where $C^{i}_{U,B}$ is the usual Cartier operator \cite[Theorem 7.2]{Kat70}.

Moreover, $B^{i}_U((\log B)(pD))$ and $Z^{i}_U((\log B)(pD))$ are locally free $\sO_U$-modules. 
\end{lem}
\begin{proof}
We obtain \eqref{log smooth 1} by the definitions of $B^{i}_U((\log B)(pD))$ and $Z^{i}_U((\log B)(pD))$.
Although \eqref{log smooth 2} have been constructed in \cite[Subsection 3.4]{Hara98}, we include the sketch of the proof for the completeness.
By \cite[Lemma 3.3]{Hara98}, we have an exact sequence
\[
0 \to B^{i}_U((\log B)(D'))\to Z^{i}_U((\log B)(D')) \xrightarrow{(C^{i}_{U,B})'} \Omega^{i}_U(\log B) \to 0,
\]
such that the map $(C^{i}_{U,B})'$ coincides with the usual Cartier operator $C^i_{U,B}$ on $U\setminus \Supp(D')$ for all $i\geq 0$.
By tensoring with $\sO_U(D)$, we obtain the exact sequence
\[
0 \to B^{i}_U((\log B)(pD))\to Z^{i}_U((\log B)(pD)) \xrightarrow{C^{i}_{U,B}(D)} \Omega^{i}_U(\log B)(D) \to 0,
\]
where $C^{i}_{U,B}(D)\coloneqq (C_{U,B}^{i})'\otimes \sO_U(D)$.
Then $C^{i}_{U,B}(D)$ coincides with $C^{i}_{U,B}\otimes \sO_U(D)$ on $U\setminus \Supp(D')$.

Next, we prove that $B^{i}_U((\log B)(pD))$ and $Z^{i}_U((\log B)(pD))$ are locally free.
We note that if the last term of a short exact sequence is locally free, then the sequence splits locally, and if we assume that the middle term is locally free in addition, then so is the first term.
We also note that the Frobenius pushforward of a locally free $\sO_U$-module is locally free by Kunz's theorem.
Therefore, if $Z^{i}_U((\log B)(pD))$ is locally free, then so is $B^{i}_U((\log B)(pD))$ by \eqref{log smooth 2}, and
if $B^{i+1}_U((\log B)(pD))$ is locally free, then so is $Z^{i}_U((\log B)(pD))$ by \eqref{log smooth 1}.
Now, since $Z^{\dim d}_U((\log B)(pD))=F_{*}\omega_U(B+pD)$ is locally free,
we conclude.
\end{proof}

\begin{rem}\label{rem:tensoring Z-divs}
If $D$ is a $\Z$-divisor, then $\Supp(D')=\emptyset$ and 
\begin{align*}
    &B^{i}_U((\log B)(pD))= B^{i}_U(\log B)\otimes \sO_U(D)\\
    &Z^{i}_U((\log B)(pD))= Z^{i}_U(\log B)\otimes \sO_U(D),\,\,\text{and}\\
    &C^{i}_{U,B}(D)=C^{i}_{U,B}\otimes \sO_U(D)
\end{align*}
hold for all $i\geq 0$ by Lemma \ref{lem:Cartier isomorphism (log smooth case)}.
\end{rem}

\begin{defn}\label{def:reflexive Carter operators}
We define reflexive $\sO_X$-modules by 
\begin{align*}
    &B^{[i]}_X((\log B)(pD))\coloneqq j_{*}B^{i}_U((\log B)(pD))\,\,\text{and}\\
    &Z^{[i]}_X((\log B)(pD))\coloneqq j_{*}Z^{i}_U((\log B)(pD))
\end{align*}
for all $i\geq 0$.
The \textit{$i$-th reflexive Cartier operator}
\[
C^{[i]}_{X,B}(D)\colon Z^{[i]}_X((\log B)(pD))\to \Omega^{[i]}_X(\log B)(D)
\]
\textit{associated to $D$} is defined as $j_{*}C^{i}_{U,B}(D)$ for all $i\geq 0$.
We simply call $C^{[i]}_{X,B}\coloneqq C^{[i]}_{X,B}(0)$ the \textit{$i$-th reflexive Cartier operator}.
\end{defn}

\begin{lem}\label{lem:Cartier operators}
There exist the following exact sequences:
\begin{equation}
0 \to Z^{[i]}_X((\log B)(pD)) \to F_{*}\Omega^{[i]}_X(\log B)(pD) \to B^{[i+1]}_X((\log B)(pD)),\label{non-log smooth 1}
\end{equation}
\begin{equation}
0 \to B^{[i]}_X((\log B)(pD))\to Z^{[i]}_X((\log B)(pD))\xrightarrow{C^{[i]}_{X,B}(D)} \Omega^{[i]}_X(\log B)(D),\label{non-log smooth 2}
\end{equation}
for all $i\geq 0$.
Moreover, \eqref{non-log smooth 1} and \eqref{non-log smooth 2} are exact on the right on $U$, and the map $C^{[i]}_{X,B}(D)|_{U}$ coincides with $C^{i}_{U,B}\otimes\sO_U(D)$ on $U\setminus \Supp(D')$.
\end{lem}
\begin{proof}
Taking the pushforward of \eqref{log smooth 1} and \eqref{log smooth 2} by $j$, we obtain the assertions.
\end{proof}

\begin{rem}\label{rem:ignoring boundary}
If $D$ is a $\Z$-divisor, then the reflexive $\sO_X$-module
\[
B^{[1]}_X((\log B)(pD))=j_{*}(\mathrm{Im}(F_{*}d\otimes\sO_U(D) \colon F_{*}\sO_U(pD) \to F_{*}\Omega^{1}_U(\log B)(pD)))
\] depends only on $D$, not $B$.
In this case, we simply write $B^{[1]}_X((\log B)(pD))$ as $B^{[1]}_X(pD)$.
\end{rem}

\begin{rem}
Taking $i=0$ and $D=0$ in \eqref{non-log smooth 1}, we obtain an exact sequence
\[
0 \to \sO_X \to F_{*}\sO_X \to B_X^{[1]},
\]
and the first map is the Frobenius map.
\end{rem}

\subsection{Surjectivity of the reflexive Cartier operator}\label{Sec:Suj of reflexive Cartier operators}

In this subsection, we focus on the surjectivity of the reflexive Cartier operator.

\begin{lem}\label{lem:exactness of Cartier Operator on F-lift}
Let $X$ be a normal variety and $B$ a reduced divisor on $X$ such that $(X,B)$ is locally $F$-liftable.
Then the $i$-th reflexive Cartier operator 
\[
C^{[i]}_{X,B}(D)\colon Z_X^{[i]}((\log B)(pD))\to \Omega_X^{[i]}(\log B)(D)
\]
associated to $D$ is a locally split surjection for all $i\geq 0$ and for every $\Z$-divisor $D$ on $X$.
\end{lem}
\begin{proof}
Since the assertion is local on $X$, we may assume that $(X,B)$ is $F$-liftable.
Let $j\colon U\hookrightarrow X$ be the inclusion of the log smooth locus of $(X,B)$.
By \cite[Proposition 3.2.1 and Variant 3.2.2]{AWZ}, the $i$-th Cartier operator
\[
C_{U,B}^{i}\colon Z^i_U(\log B)\to\Omega_U^{i}(\log B)
\] 
is a split surjection for all $i\geq 0$.
We take a $\Z$-divisor $D$ on $X$. 
Since $C^{i}_{U,B}(D)=C^i_{U,B}\otimes \sO_U(D)$ by Remark \ref{rem:tensoring Z-divs},
the map
\[
C^{i}_{U,B}(D)\colon Z^i_U((\log B)(pD))\to \Omega_U^{i}(\log B)(D)
\] 
is a split surjection. 
Taking the pushforward by $j$, we obtain the assertion.
\end{proof}

\begin{lem}\label{lem:B:S_3toC:surj}
Let $X$ be a normal variety, $B$ a reduced divisor on $X$, and $Z$ the non-log smooth locus of $(X,B)$. Let $i\geq 0$ be an integer. 
Suppose that $Z$ has codimension at least three and $B_X^{[i]}(\log B)$ satisfies $(S_3)$. Then $C_{X,B}^{[i]}\colon Z_X^{[i]}(\log B)\to \Omega_X^{[i]}(\log B)$
is surjective.
\end{lem}
\begin{proof}
By \eqref{non-log smooth 2}, we have the exact sequence
\[ 0\to B^{[i]}_X(\log B)\to Z^{[i]}_X(\log B)\xrightarrow{C_{X,B}^{[i]}} \Omega^{[i]}_X(\log B).\]
Since $B^{[i]}_X(\log B)$ satisfies $(S_3)$, we obtain $\mathrm{depth}_{\m_x}(B^{[i]}_X(\log B)_{x})\geq 3$ for every $x\in Z$.
Similarly, since $Z^{[i]}_X(\log B)$ satisfies $(S_2)$, we have $\mathrm{depth}_{\m_x}(Z^{[i]}_X(\log B)_{x})\geq 2$ for every $x\in Z$.
Then the depth lemma \cite[Lemma 2.60]{Kol13} shows that $\mathrm{depth}_{\m_x}(\mathrm{Im}(C_{X,B}^{[i]})_{x})\geq 2$ for every $x\in Z$. Since $\mathrm{Im}(C_{X,B}^{[i]})$ coincides with $\Omega^{[i]}_X(\log B)$ outside $Z$, it follows that $\mathrm{Im}(C_{X,B}^{[i]})$ satisfies $(S_2)$ and $\mathrm{Im}(C_{X,B}^{[i]})=\Omega^{[i]}_X(\log B)$, as desired.
\end{proof}

\begin{lem}\label{prop:surjectivity of higher codimension case}
Let $X$ be a normal variety, $B$ a reduced divisor on $X$, and $n\geq 2$ an integer.
Suppose that one of the following conditions holds:
\begin{enumerate}
    \item[\textup{(1)}] The singular locus of $X$ has codimension at least $n+1$ and $X$ satisfies $(S_{n+1})$.
    \item[\textup{(2)}] $X$ is $F$-injective and $X$ satisfies $(S_n)$.
\end{enumerate}
Then $B_X^{[1]}$ satisfies $(S_n)$.
In particular, if $n=3$ and the non-log smooth locus of $(X,B)$ has codimension at least three, then 
$C_{X,B}^{[1]}\colon Z_X^{[1]}(\log B)\to \Omega_X^{[1]}(\log B)$ is surjective.
\end{lem}
\begin{proof}
We prove that $B_X^{[1]}=B_X^{[1]}(\log B)$ satisfies $(S_n)$. Then the last assertion is an immediate consequence of Lemma \ref{lem:B:S_3toC:surj}. 
We consider the exact sequence 
\[
0\to \sO_X \to F_{*}\sO_X \to B_X^{[1]},
\]
and let $\mathcal{B}\coloneqq \mathrm{Im}(F_{*}\sO_X \to B_X^{[1]})$.
Suppose that (1) holds. Then the depth lemma shows that $\mathcal{B}$ satisfies $(S_n)$. Therefore, $B_X^{[1]}=\mathcal{B}$ and $B_X^{[1]}$ satisfies $(S_{n})$.
Next, suppose that (2) holds. We take a point $x\in X$ and consider the exact sequence
\[
H^{r}_{\m_x}(\sO_{X,x})\to H^{r}_{\m_x}(\mathcal{B}_{x}) \to H^{r+1}_{\m_x}(\sO_{X,x}) \xrightarrow{F} H^{r+1}_{\m_x}(\sO_{X,x})
\]
for all $r\geq 0$.
Since the third map is injective by assumption, we obtain $\mathrm{depth}_{\m_x}(\mathcal{B}_{x})\geq \mathrm{depth}_{\m_x}(\sO_{X,x})$, 
and thus $\mathcal{B}$ satisfies $(S_n)$. 
Therefore, $B_X^{[1]}=\mathcal{B}$ and $B_X^{[1]}$ satisfies $(S_{n})$.
\end{proof}

\section{Extension theorem}\label{sec:ext thm}

\subsection{Proof of Theorems \ref{Introthm:ext thm}, \ref{Introthm:linearly reductive quotient}, and \ref{Introthm:Flenner type ext thm}}
The surjectivity of the reflexive Cartier operator shows the logarithmic extension theorem.

\begin{thm}\label{thm:generalized ext thm}
Let $X$ be a normal variety, $B$ a reduced divisor on $X$, and $D$ a $\Q$-Cartier $\Z$-divisor on $X$.
Let $f\colon Y\to X$ be a proper birational morphism from a normal variety $Y$.
We fix an integer $i\geq 0$.
Suppose that the $i$-th reflexive Cartier operator 
\[
C^{[i]}_{X,B}(p^{n-1}D)\colon Z_X^{[i]}((\log B)(p^nD)) \to \Omega_X^{[i]}(\log B)(p^{n-1}D)
\]
associated to $p^{n-1}D$ is surjective for all $n>0$.
Then the restriction map
\[
f_{*}\Omega^{[i]}_Y(\log B_Y)(f^{*}D)\hookrightarrow\Omega_X^{[i]}(\log B)(D).
\]
is an isomorphism, where $B_Y\coloneqq f_{*}^{-1}B+E$ and $E$ is the reduced $f$-exceptional divisor.
\end{thm}
\begin{proof}
Since the assertion is local on $X$, we may assume that $X$ is affine and $D$ is torsion.
Take an effective $\Q$-divisor $A_Y\geq 0$ such that $\Supp(A_Y)=E$ and $\lfloor f^{*}D+A_Y \rfloor=\lfloor f^{*}D \rfloor$.
Since $\{p^rf^{*}D\mid r\geq 0\}$ is a finite set, 
we can take $n\gg0$ so that, for every irreducible component of $E$, its coefficient in $p^nA_Y$ is bigger than the pole of rational sections of $\Omega^{[i]}_Y(\log B_Y)(p^{r}f^{*}D)$ for all $r\geq 0$.
In particular, by taking $r=n$, we have an isomorphism
\[
f_{*}\Omega^{[i]}_Y(\log B_Y)(p^n(f^{*}D+A_Y))\cong\Omega_X^{[i]}(\log B)(p^nD).
\]
By \eqref{non-log smooth 1}, we have the following exact sequence:
\begin{multline*}
0\to Z^{[i]}_Y((\log B_Y)(p^n(f^{*}D+A_Y)))\to F_{*}\Omega^{[i]}_Y(\log B_Y)(p^n(f^{*}D+A_Y))\\
\to B^{[i+1]}_Y((\log B_Y)(p^n(f^{*}D+A_Y))).
\end{multline*}
Taking the pushforward by $f$, we obtain the following exact sequence:
\begin{multline*}
0\to f_{*}Z^{[i]}_Y((\log B_Y)(p^n(f^{*}D+A_Y)))\to F_{*}\Omega_X^{[i]}(\log B)(p^nD)\\\to f_{*}B^{[i+1]}_Y((\log B_Y)(p^n(f^{*}D+A_Y))).
\end{multline*}
Since the middle term is reflexive and the last term is torsion-free, the first term is reflexive.
Thus, the restriction map 
\[
f_{*}Z^{[i]}_Y((\log B_Y)(p^n(f^{*}D+A_Y)))\hookrightarrow  Z_X^{[i]}((\log B)(p^nD))
\]
is an isomorphism.
By \eqref{non-log smooth 2}, we have the following commutative diagram:
\[
\xymatrix@C=60pt{
 f_{*}Z^{[i]}_Y((\log B_Y)(p^n(f^{*}D+A_Y)))\ar[r]^{f_{*}C^{[i]}_{Y,B_Y}(p^{n-1}(f^{*}D+A_Y))}\ar[d]^{\cong} & f_{*}\Omega^{[i]}_Y(\log B_Y)(p^{n-1}(f^{*}D+A_Y))\ar@{^{(}->}[d] \\
  Z_X^{[i]}((\log B)(p^nD))\ar@{->>}[r]^{C^{[i]}_{X,B}(p^{n-1}D)}& \Omega_X^{[i]}(\log B)(p^{n-1}D), \\
}
\]
where both vertical maps are the restriction maps.
Since the lower horizontal map is surjective by assumption, 
it follows that 
\[
f_{*}\Omega^{[i]}_Y(\log B_Y)(p^{n-1}(f^{*}D+A_Y))\cong\Omega_X^{[i]}(\log B)(p^{n-1}D).
\]
Repeating this procedure, we conclude that  
\[
f_{*}\Omega^{[i]}_Y(\log B_Y)(f^{*}D) =f_{*}\Omega^{[i]}_Y(\log B_Y)(f^{*}D+A_Y) \cong\Omega_X^{[i]}(\log B)(D),
\]
as desired.
\end{proof}

\begin{proof}[Proof of Theorem \ref{Introthm:ext thm}]
Taking $D=0$ in Theorem \ref{thm:generalized ext thm}, we obtain the assertion.
\end{proof}

\begin{thm}\label{thm:ext thm for F-lift}
Let $X$ be a normal variety and $B$ a reduced divisor on $X$ such that $(X, B)$ is locally $F$-liftable.
Let $D$ be a $\Q$-Cartier $\Z$-divisor on $X$ and $f\colon Y\to X$ a proper birational morphism from a normal variety $Y$.
Then the restriction map
\[
f_{*}\Omega^{[i]}_Y(\log B_Y)(f^{*}D)\hookrightarrow\Omega_X^{[i]}(\log B)(D).
\]
is an isomorphism for all $i\geq 0$, where $B_Y\coloneqq f_{*}^{-1}B+E$ and $E$ is the reduced $f$-exceptional divisor.
In particular, $(X,B)$ satisfies the logarithmic extension theorem.
\end{thm}
\begin{proof}
The assertion follows from Lemma \ref{lem:exactness of Cartier Operator on F-lift} and Theorem \ref{thm:generalized ext thm}.
\end{proof}

\begin{rem}
Extendability in the type of Theorems \ref{thm:generalized ext thm} and \ref{thm:ext thm for F-lift} is useful especially when the Cartier index of $D$ is divisible by $p$. 
We refer to \cite[Remark 4.10]{Kaw3} for details.
\end{rem}

\begin{proof}[Proof of Theorem \ref{Introthm:linearly reductive quotient}]
The assertions follow from Theorems \ref{thm:AWZ}, \ref{thm:F-liftable singularities}, and \ref{thm:ext thm for F-lift}.
\end{proof}

\begin{proof}[Proof of Theorem \ref{Introthm:Flenner type ext thm}]
The assertion follows from Lemma \ref{prop:surjectivity of higher codimension case} and Theorem \ref{Introthm:ext thm}.
\end{proof}

In the next proposition, we see a case where the surjectivity of the reflexive Cartier operator is a necessary condition for the logarithmic extension theorem.

\begin{prop}\label{prop:Converse direction}
Let $X$ be a normal variety.
Suppose that there exists a proper birational morphism $f\colon Y\to X$ from a smooth variety $Y$ such that the reduced $f$-exceptional divisor $E$ is simple normal crossing and $R^{i}f_{*}\sO_Y=0$ for all $i\in\{1,2\}$.
Then $C^{[1]}_{X}\colon Z^{[1]}_{X}\to \Omega_X^{[1]}$ is surjective if and only if $X$ satisfies the logarithmic extension theorem for one-forms.
\end{prop}
\begin{proof}
By Theorem \ref{Introthm:ext thm}, it suffices to show the if direction.
By assumption and the exact sequence 
\[
0\to \sO_Y \to F_{*}\sO_Y \to B_{Y}^{1}\to 0,
\]
we obtain $R^1f_{*}B_{Y}^{1}=0$.
By \eqref{log smooth 2} and Remark \ref{rem:ignoring boundary}, we have an exact sequence
\[
0\to B_{Y}^{1} \to Z^{1}_Y(\log E)\xrightarrow{C^1_{Y,E}} \Omega^{1}_Y(\log E)\to 0.
\]
Thus, $f_{*}C^1_{Y,E}\colon f_{*}Z^{1}_Y(\log E)\to f_{*}\Omega^{1}_Y(\log E)$ is surjective.
Since we have the following commutative diagram:
\[
\xymatrix@C=80pt{
 f_{*}Z^{1}_Y(\log E)\ar@{->>}[r]^{f_{*}C^{1}_{Y,E}}\ar@{^{(}->}[d] & f_{*}\Omega^{1}_Y(\log E)\ar[d]^{\cong} \ar[d]\\
  Z_X^{[1]}\ar[r]^{C_{X}^{[1]}}& \Omega_X^{[1]},\\
}
\]
we conclude that $C_{X}^{[1]}\colon Z^{[1]}_{X}\to \Omega_X^{[1]}$ is surjective.
\end{proof}

The next example shows that a quotient singularity by a non-linearly reductive group scheme violates the logarithmic extension theorem, even if the group scheme is reduced.

\begin{eg}\label{example:C_2-quotient}
    Let $k$ be an algebraically closed field of characteristic two.
    A rational double point of type $E_8^2$ over $k$ is a quotient singularity by the cyclic group scheme $\mathbf{C_2}$, as shown in \cite[Table 4]{LMM2}.
    It can be observed that a normal affine saurface $X$ over $k$ that has only such a singularity violates the logarithmic extension theorem as follows:

    Let $f\colon Y\to X$ be the minimal resolution. Since $\omega_X\cong \sO_X$ and $\omega_Y\cong \sO_Y$, the tangent sheaves $T_X$ and $T_Y$ are isomorphic to $\Omega^{[1]}_X$ and $\Omega^{1}_Y$ respectively. Then \cite[Theorem 1.1 (iii)]{Hirokado} shows that $\dim_{k}\Omega_X^{[1]}/f_{*}\Omega^{1}_Y=2$, and in particular, $X$ violates the regular extension theorem. Since the determinant of the intersection matrix of the reduced $f$-exceptional divisor is equal to one, it follows from \cite[8.C]{Gra} that $X$ violates the logarithmic extension theorem.

    In particular, the linearly reductive property in Theorem \ref{Introthm:linearly reductive quotient} (1) and (2) is necessary.
\end{eg}

Graf \cite[Theorem 3]{Graf21} constructed counterexamples to Flenner's extension theorem for $i$-forms for all $i\geq 0$ in positive characteristic assuming the existence of resolutions of singularities.
On the other hand, we do not need resolutions of singularities at least for counterexamples to the theorem for one-forms.

\begin{eg}\label{eg:counterex to Flenner:non-Cohen-Macaulay case}
Let $p$ be a prime number and $d\geq 3$ an integer.
Mukai constructed a smooth projective variety $X$ over a perfect field of characteristic $p$ of $\dim X=d-1$ and an ample Cartier divisor $A$ on $X$ such that $H^0(X, \Omega_X^{1}(-A))\neq 0$
(see \cite[Theorem 2 and Corollary in p.~519]{Mukai}).
Then the singular point of the affine cone \[C_{a}(X,A)=\Spec \bigoplus_{n\geq 0} H^0(X, \sO_X(nA))\] over $X$ has codimension $d$ and $C_{a}(X,A)$ violates the logarithmic extension theorem for one-forms by \cite[Proposition 11.3]{Gra}.
\end{eg}

\subsection{Proof of Theorem \ref{Introthm:tame quotient}}

We recall the \textit{trace map} of sheaves of differential forms (\cite[Proposition 2.2.23]{CR11}, \cite{Garel}, \cite[Section 16]{Kunz}) and its reflexive version (\cite[Proof of Theorem C]{Kawakami-Totaro}).

\begin{defn}\label{def:reflexive trace map}
    Let $g\colon X'\to X$ be a finite surjective morphism of normal varieties.
    Let $U\subset X$ be an open subscheme such that $U$ and $U'\coloneqq g^{-1}(U)$ are contained in the smooth loci of $X$ and $X'$ respectively, and $X\setminus U$ and $X'\setminus U'$ have codimension at least two.
    By abuse of the notation, we also denote $g|_{U'}$ by $g$.
    The \textit{$i$-th trace map} $\tau_{g}^{i}\colon g_{*}\Omega^i_{U'}\to \Omega^i_U$ is defined as 
    \begin{align*}
      g_{*}\Omega^{i}_{U'}\cong g_{*}\mathcal{H}\!\mathit{om}_{\sO_{U'}}(\Omega^{d-i}_{U'},\omega_{U'})&\xrightarrow{\mathrm{nat.}} \mathcal{H}\!\mathit{om}_{\sO_{U}}(g_{*}\Omega^{d-i}_{U'},g_{*}\omega_{U'})\\
      &\xrightarrow{\mathrm{Tr_g}}\mathcal{H}\!\mathit{om}_{\sO_U}(g_{*}\Omega^{d-i}_{U'},\omega_{U})\\
      &\xrightarrow{\circ g^{*}}\mathcal{H}\!\mathit{om}_{\sO_U}(\Omega^{d-i}_{U},\omega_{U})
      \cong \Omega_U^{i},
    \end{align*}
    for all $i\geq 0$, where $d\coloneqq \dim X$ and $\mathrm{Tr}_{g}\colon g_{*}\omega_{U'}\to \omega_U$ denotes the usual trace map.  
    By construction, the map $\tau_{g}^{d}$ coincides with $\mathrm{Tr}_{g}$.
    The \textit{$i$-th reflexive trace map} $\tau_{g}^{[i]}\colon g_{*}\Omega^{[i]}_{X'}\to \Omega^{[i]}_X$ is defined as $j_{*}\tau_{g}^{i}$, where $j\colon U\hookrightarrow X$ is the inclusion.
\end{defn}

\begin{lem}[\textup{cf.~\cite[Proof of Theorem C]{Kawakami-Totaro}}]\label{lem:reflexive trace map}
    Let $g\colon X'\to X$ be a finite surjective morphism of normal varieties of degree prime to $p$.
    Then the $i$-th reflexive trace map \[\tau_{g}^{[i]}\colon g_{*}\Omega_{X'}^{[i]}\to\Omega_X^{[i]}\] 
    is a split surjection for all $i\geq 0$.
\end{lem}
\begin{proof}
    We use the notation of Definition \ref{def:reflexive trace map}.
    By \cite[Proposition 2.2.23 (3)]{CR11}, the composition $\tau_{g}^{i}\circ g^{*}\colon \Omega^i_U\to \Omega_U^i$ is the multiplication by $\deg(g)$ for all $i\geq 0$.
    Since $\deg(g)$ is not divisible by $p$, it follows that $\tau_{g}^{i}\circ \frac{1}{\deg(g)}g^{*}$ is the identity.
    Taking the pushforward by the inclusion $j\colon U\hookrightarrow X$, we obtain the assertion.
\end{proof}

\begin{thm}\label{thm:tame quotients}
Let $g\colon X'\to X$ be a finite surjective morphism of normal varieties of degree prime to $p$.
If $X'$ satisfies the regular extension theorem, then so does $X$.
    \end{thm}
\begin{proof}
Let $f\colon Y\to X$ be a proper birational morphism.
Let $Y'$ be the normalization of the component of the fiber product dominating $X'$, and let $f'\colon Y'\to X'$ and $g_Y\colon Y'\to Y$ be compositions of the normalization and projections.

Since $\deg(g)=\deg(g_Y)$ is not divisible by $p$, the $i$-th reflexive trace maps $\tau_{g}^{[i]}$ and $\tau_{g_Y}^{[i]}$ are split surjections for all $i\geq 0$ by Lemma \ref{lem:reflexive trace map}.
By the construction of the reflexive trace map in Definition \ref{def:reflexive trace map}, we have the following commutative diagram:
\[
\xymatrix@C=80pt{
f_{*}(g_Y)_{*}\Omega^{[i]}_{Y'}\ar[d]\ar@{->>}[r]^-{f_{*}\tau^{[i]}_{g_Y}} & f_{*}\Omega^{[i]}_Y\ar@{^{(}->}[d]\\
 g_{*}\Omega^{[i]}_{X'}\ar@{->>}[r]^-{\tau^{[i]}_{g}} & \Omega_X^{[i]}.\\
}
\]
Suppose that $X'$ satisfies the regular extension theorem.
Then we have an isomorphism
\[f_{*}(g_Y)_{*}\Omega^{[i]}_{Y'}=g_{*}(f')_{*}\Omega^{[i]}_{Y'}\cong g_{*}\Omega^{[i]}_{X'}.\] 
Therefore, we conclude from the above diagram that the restriction map $f_{*}\Omega_Y^{[i]}\hookrightarrow \Omega_X^{[i]}$ is an isomorphism.
\end{proof}

\begin{lem}\label{lem:ext thm for S_2 differential}
Let $i\geq 0$ be an integer.
Let $X$ be a normal variety such that $\Omega^{i}_X$ is reflexive.
Then $X$ satisfies the regular extension theorem for $i$-forms.
\end{lem}
\begin{proof}
Let $j\colon U\hookrightarrow X$ be the inclusion of the smooth locus.
Let $f\colon Y\to X$ be a proper birational morphism from a normal variety $Y$.
Then we have a composition of natural maps
\[
\Omega_X^{i}\xrightarrow{f^{*}} f_{*}\Omega^{i}_Y\rightarrow f_{*}\Omega^{[i]}_Y \hookrightarrow \Omega_X^{[i]},
\]
which coincides with the pullback map $j^{*}\colon\Omega^{i}_X\hookrightarrow j_{*}\Omega_{U}^{i}$, and thus is an isomorphism by assumption.
Therefore, the restriction map $f_{*}\Omega^{[i]}_Y \hookrightarrow \Omega_X^{[i]}$ is an isomorphism.
\end{proof}
\begin{rem}
Let $X$ be a normal locally complete intersection variety and $i\geq 0$ an integer.
If the singular locus of $X$ has codimension at least $i+2$, then $\Omega_X^{i}$ is reflexive by \cite[Lemma 3.11]{Sato-Takagi}.
Therefore, by Lemma \ref{lem:ext thm for S_2 differential}, $X$ satisfies the regular extension theorem for $i$-forms.
This shows that Flenner's extension theorem \cite{Flenner88} holds for normal locally complete intersection varieties over a perfect field of positive characteristic.
\end{rem}

\begin{proof}[Proof of Theorem \ref{Introthm:tame quotient}]
  The assertion follows from Theorem \ref{thm:tame quotients} and Lemma \ref{lem:ext thm for S_2 differential}.
\end{proof}

\section*{Acknowledgements}
The author expresses his gratitude to Adrian Langer, Teppei Takamatsu, Burt Totaro, Jakub Witaszek, and Shou Yoshikawa for helpful conversations.
He would like to thank the anonymous referee for valuable suggestions.
He is also grateful to Yuya Matsumoto, Hiromu Tanaka, and Fuetaro Yobuko for their useful discussion and comments.
This work was supported by JSPS KAKENHI Grant number JP22J00272.

\input{main.bbl}


\end{document}

%% file: main.bbl
\newcommand{\etalchar}[1]{$^{#1}$}